\documentclass[10pt,reqno]{amsart}

\usepackage{mathrsfs}

\usepackage{amsthm}
\usepackage{amssymb}

  \theoremstyle{plain}
    \newtheorem{theorem}{Theorem}[section]

    \newtheorem*{proposition*}{Proposition}
    
    \newtheorem*{corollary*}{Corollary}
    
    \newtheorem*{lemma*}{Lemma}
    \newtheorem{conjecture}[theorem]{Conjecture}
    
  \theoremstyle{definition}
    \newtheorem{definition}[theorem]{Definition}
    \newtheorem*{definition*}{Definition}
    
  \theoremstyle{remark}
    \newtheorem*{remark}{Remark}
    
\usepackage[usenames,dvipsnames]{xcolor}

  \newcommand{\bbC}{\mathbb{C}}

  \newcommand{\bbR}{\mathbb{R}}

  \newcommand{\bbV}{\mathbb{V}}

  \newcommand{\bbZ}{\mathbb{Z}}

  \newcommand{\scrA}{\mathscr{A}}
  \newcommand{\scrB}{\mathscr{B}}

  \newcommand{\scrI}{\mathscr{I}}

  \newcommand{\scrU}{\mathscr{U}}
  \newcommand{\scrV}{\mathscr{V}}
  
  \newcommand{\scrX}{\mathscr{X}}

    \DeclareMathOperator{\Sym}{Sym}

\title{A Finiteness Property of Torus Invariants}
\author{Stella Gastineau and Samuel Tenka}

\usepackage{bm}
\begin{document}

\maketitle
\begin{abstract}
In this paper the invariant subring $R_n$ of an algebraic 
torus $T=(\mathbb{C}^\times)^r$ acting
on the multi-homogenous polynomial ring
    $$S^{\boxtimes n}=\bigoplus_{d=0}^\infty (S^{(d)})^{\otimes n},$$
where $S^{(d)}$ is the $d$th graded piece of the polynomial 
ring $S=\mathbb{C}[x_1,\dots,x_k]$, 
is studied from the viewpoint of matrices whose entries sum to zero.
Using these weight matrices we prove that there exists a 
$d_1$ such that for all positive integers $n$, 
the relations of the invariant subring $R_n$ are generated 
in multi-homogenous degree $\leq d_1$.\\
{\bf Grant:} 0943832
\end{abstract}

\section{introduction}
Let $G$ be a reductive group and $S$ a finitely generated 
graded algebra over $\bbC$ concentrated in non-negative degrees 
on which $G$ acts. For all positive integers
$n$ let
\[
	R_n=\bigoplus_{k=0}^\infty R_n^{(k)},\quad\mbox{where } R_n^{(k)}=
	((S^{(k)})^{\otimes n})^G,
\]
where $S^{(k)}$ is the $k$th graded piece of $S$. Note that $R_n$ is a 
$\bbC$-algebra and so is finitely generated. 
The following uniformity conjecture for
the sequence of algebras $\{R_n\}$ was made by Andrew Snowden in 
an unpublished paper \cite{snow}:
\begin{conjecture}[Snowden]\label{conj}
	Keep the above notation.
	\begin{enumerate}
	\item 
		Generators --- There exists a positive integer $d$ such that each algebra $R_n$
		is generated in degrees $\leq d$.
	\item
		Relations --- Let $d$ be as above. For all positive integers $n$ let
		\[
			P_n=\Sym\left(\bigoplus_{k=0}^dR_n^{(k)}\right).
		\]
		There exists a positive integer $s$ such that the kernel of the surjection 
		$P_n\to R_n$ is generated in degrees $\leq s$ as a $P_n${\rm -}module for all $n$.
	\item
		Syzygyies --- There exist integers $s_0,s_1,\dots$ such that
		for each $n$ there exists a resolution
		\[
			\cdots\to F_{n,2}\to F_{n,1}\to F_{n,0}\to R_n\to0
		\]
		of $R_n$ by finite $P_n${\rm -}modules $F_{n,i}$ such that 
		$F_{n,i}$ is generated in degrees $\leq s_i$.
	\end{enumerate}
\end{conjecture}

Jonathan Brito proved in \cite{brito}
that Conjecture~\ref{conj}(a) is true when $G$ 
is a complex $r$-torus $T=(\bbC^{\times})^r$ and $S$ is a polynomial ring 
$\Sym V=\bbC[x_1,\dots,x_k]$ where $V$ is a $k$ dimensional polynomial
representation of $T$. 
The purpose of this paper is to prove Conjecture~\ref{conj}(b) for this
specific case. A formal statement of this result is given in
Theorem~\ref{main theorem} (see page~\pageref{main theorem}).

\section{Torus Invariants}
\subsection*{Multi-homogeneous monomials}
Fix a complex $r$-torus $T=(\bbC^\times)^r$ 
and $k$-dimensional polynomial representation $V$ of $T$.
Suppose that the representation $V$ is given by the weight vectors
$w_1,\dots,w_k\in\bbZ^r$. We will denote the set of these weights
by $\scrX$.
Then for any positive integer $n$, the action of $T$ on $V$
naturally extends to an action of $T$ on the
{\em multi-homogeneous polynomial ring}
\[
	S^{\boxtimes n}=\bigoplus_{d=0}^\infty(\Sym^d V)^{\otimes n}
\]
in the following way:
Suppose that 
$\bbV=\{x_1,\dots,x_k\}$ is a basis of $V$ such that for each $i$,
\[
	t\cdot x_i=t^{w_i}x_i=t_1^{w_{i1}}\cdots t_r^{w_{ir}}x_i.
\]
For simplicity we will write $|x_i|$ rather than $w_i$. A pure tensor in $S^{\boxtimes n}$
is a {\em multi-homogeneous
monomial of degree $d$} if each tensor factor is a monomial in $\bbV$ of degree 
$d$. For every multi-homogeneous monomial
\begin{equation}\label{multi-hom}
	f=f_1\otimes\cdots\otimes f_n,\quad\mbox{where }
	f_i=\phi_{i1}\cdots \phi_{id}
\end{equation}
of degree $d$ let
\begin{equation}\label{action}
	t\cdot f=t^{|\phi_{11}|+\cdots+|\phi_{nd}|}f.
\end{equation}

The sum in \eqref{action} is very unorganized. For example, 
it remains the same if you permute the tensor products or if 
you swap out one basis element for another
with an equal weight. The notion of weight matrices given below in 
Definition~\ref{weight matrices} helps organize this sum in a more stuctured
manner.

\begin{definition}\label{weight matrices}
	Suppose $f$ is a multi-homogeneous monomial of degree $d$
	\[
		f=f_1\otimes\cdots\otimes f_n,\quad\mbox{where }
		f_i=\phi_{i1}\cdots\phi_{id}.
	\]
	Then let
	\[	
		\|f\|=
		\left[\!\!\!\begin{array}{ccc}
		|\phi_{11}| & \cdots & |\phi_{1d}| \\
		\vdots & & \vdots \\
		|\phi_{n1}| & \cdots & |\phi_{nd}|
		\end{array}\!\!\!\right].
	\]
	We will call $\|f\|$ the {\em weight matrix} associated to $f$.
\end{definition}

\begin{remark}
	Note that if two or more basis vectors have the same weight vector, then
	a single weight matrix can be associated to more than one multi-homogeneous monomials.
	Furthermore, since multiplication is commutative inside the tensor factors,
	we can permute the entries in any fixed row of $\|f\|$
	and the resulting matrix would still be a weight matrix associated with $f$.
	Since we don't want this ambiguity, we will 
	assume that weight matrices 
	remember which variables are involved 
	in the monomial and the order they are written in.
	Since the concept of a weight matrix is merely an organizational tool, 
	requiring this additional structure does not change the math involved.
\end{remark}

\subsection*{The invariant subring $\bm{R_n}$}
For all positive integers $n$ let
\[\begin{aligned}
	R_n & =(S^{\boxtimes n})^T \\
	& =\left[\bigoplus_{d=0}^\infty (\Sym^d V)^{\otimes n}\right]^T
\end{aligned}\]
be the invariant subring generated by all invariant multi-homogeneous monomials.
Here a multi-homogeneous monomial $f\in S^{\boxtimes n}$ 
of degree $d$ is invariant (with respect to $T$) if and only if 
\[
	\sum_{i=1}^n\sum_{j=1}^d|\phi_{i\!j}|=0.
\]
Equivalently, a multi-homogeneous monomial is invariant if and only
if the weight matrix associated with it is
zero-sum, i.e. its entries sum to zero.

Suppose that $d_0$ is a positive integer
such that each $R_n$ is generated by 
multi-homogeneous monomials in degrees $\leq d_0$.
A proof that such a $d_0$ exists can be found in \cite{brito}, 
which we modestly strengthen in \S\ref{s d_0}. Our main theorem is:

\begin{theorem}\label{main theorem}
	For all positive integers $n$, we will let
	\[
		P_n=\Sym\left(\bigoplus_{d=0}^{d_0}R_n^{(d)}\right).
	\]
	Then exists positive integer $d_1$ such that the kernel $\scrI_n$
	of the natural surjection $P_n\to R_n$ 
	is generated as a $P_n$-module in degrees $\leq d_1$ for all $n$.
\end{theorem}

\section{Proof of Theorem~\ref{main theorem}}
Before we can prove Theorem~\ref{main theorem}, we would like
to make some clarifications. First, we will clarify the difference
between monomials and degrees in $R_n$ and those in $P_n$. Second we will
elaborate on what relations on $R_n$ look like as elements of $\scrI_n$
and as weight matrices.

The elements of $P_n$ are polynomials
with indeterminates in the set of multi-homogeneous monomials in $R_n$
of degree $\leq d_0$.
Because of this, there are two related notions of polynomials and 
degrees. Therefore, for the sake of clarity, multi-homogeneous 
monomials in $R_n$ of degree $d\leq d_0$  will be called variables in $P_n$
with \emph{multi-homogeneous degree} $d$ and will be denoted by the lowercase
roman letters (e.g., $f$ and $g$). Here the basis elements involved
in these variables will be denoted
with the greek letters (e.g., $\phi, \gamma \in \{x_1,\dots,x_k\}$).
By contrast, the notion of monomials and degrees in $P_n$
will have the usual, understood meaning inside of a polynomial ring. 
These monomials will be denoted by the capital roman letters (e.g., $F$ and $G$).
An example of this notation is seen below in \eqref{relation}.

Next note that $\scrI_n$ is a toric variety. Then $\scrI_n$
is generated by binomials. Consider a reduced binomial $F-G\in\scrI_n$
where
\begin{equation}\label{relation}
\begin{gathered}\begin{aligned}
	F & =f^1\cdots f^p, 
		& \mbox{where }f^\nu & =f^{\nu}_{1}\otimes\cdots\otimes f^{\nu}_n
		& \mbox{and}\quad f^\nu_i &=\phi^\nu_{i1}\phi^\nu_{i2}\cdots\\
	G & =g^1\cdots g^q, 
		& \mbox{where }g^\mu & =g^\mu_1\otimes
		\cdots\otimes g^\mu_n
		& \mbox{and}\quad g^\mu_i &=\gamma^\mu_{i1}\gamma^\mu_{i2}\cdots
\end{aligned}\end{gathered}
\end{equation}
for variables $f^\nu$ and $g^\mu$ in $P_n$. Note that if we forget the order
of the basis elements in each $f^\nu$ and $g^\mu$ 
the products of the $f^\nu$ is the same in $R_n$ as the product of the $g^\mu$ 
in terms of the basis elements.
Therefore, for each fixed row of $\|F\|$, the entries can be permuted 
to form that row of $\|G\|$.
The width of these two weight matrices are the same and 
represent the total multi-homogeneous degree of the relation.

In order to prove Theorem~\ref{main theorem}, we will show we can factor this
permutation into permutations of the matrix elements from a bounded number of 
columns, with each element permuted within its row. Each factor of this permutation 
represents a sub-relation where the degree of this 
sub relation is bounded by $d_0$ times the number of entries swapped.
We'll factor the permutation by induction on the total multi-homogeneous degree of
the relation.

\begin{proof}[Proof of Theorem~\ref{main theorem}]
The first section \S\ref{s d_0} summarizes \cite{brito}
with a slight change: we are finding a bound $d_0$ given by the set $\scrX^2$ rather
than $\scrX$. We use this $d_0$ to find columns that form zero-sum
sub-matrices for our relation.
In \S\ref{s n_0} we find an analogous bound to find rows that form
zero-sum sub-matrices for our relation.

In \S\ref{s reduce} we use our bound $d_0$ to reduce the 
problem to considering binomial relations $F-G$ where $F$ and $G$
respectively contain variable $f$ and $g$ of equal multi-homogeneous degree.
We use this reduction in \S\ref{s proof} we prove that
each relation is generated in degrees $\leq n_0d_0^2$.

\subsection{}\label{s d_0}
It was proven in \cite{brito} that 
there exists a number $D$, dependent only on the representation $(\rho,V)$ of $T$, 
such that given any zero-sum matrix
with entries in $\scrX$, we can rearrange the matrix entries within their rows such that the Euclidean norm of all the column sums of the resulting matrix are 
bounded above by $D$.

Let $\scrA$ be the set of all vectors in $\bbZ^{2r}$ which are linear combinations
of vectors in $\scrX^2$ and of Euclidean norm at most $2D$.
This set is finite, so we can write $\scrA=\{a_1,\dots,a_L\}$, which defines
a rational polyhedral cone
\[
	\sigma=\left\{(\lambda_1,\dots,\lambda_L)\in\bbR^L\,\Big|\,\sum_{i}\lambda_ia_i=0,
	\mbox{and each } \lambda_i\geq 0\right\}.
\]
Let $\Lambda=\sigma\cap\bbZ^L$. Then
$\Lambda$ is finitely generated as a semi-group (see Gordon's Lemma \cite{fult}).
Fix a finite generating set $S$ of $\Lambda$. Then define
\begin{equation}\label{d_0}
	d_0=\max\big\{\lambda_{1}+\dots+\lambda_{L}
	~\big|~(\lambda_1,\dots,\lambda_L)\in S\big\}.
\end{equation}

\subsection{}\label{s reduce}
Fix a positive integer $n$ and 
a reduced binomial relation $F-G\in\scrI_n$ as given in
\eqref{relation} such that 
the Euclidean norms of the column sums of $\|f^\nu\|$ and $\|g^\mu\|$
are bounded above by $D$ given in \S\ref{s d_0}, 
and so the the Euclidean norms of the column
sums of $\|F\|$ and $\|G\|$ are bounded above by $2D$.
 
Let $M$ be the matrix with entries in $\scrX^2$ whose $i\!j$-th entry
is the concatenation of the $i\!j$-th entries of 
$\|F\|$ and $\|G\|$ respectively.
The matrices $\|F\|$ and $\|G\|$ are both zero-sum, and so
the Euclidean norms of the column sums of $M$ are bounded above by $D$. Therefore, there exist $\leq d_0$  columns $\{j_1,\dots,j_d\}$ that form a
zero-sum sub-matrix of $M$. Suppose that each $j_\ell$-th column corresponds 
to the multi-homogeneous monomials $\phi^{\nu_\ell}_{1v_\ell}\otimes\cdots\otimes
\phi^{\nu_\ell}_{nv_\ell}$ and $\gamma^{\mu_\ell}_{1u_\ell}
\otimes\cdots\otimes\gamma^{\mu_\ell}_{nu_\ell}$ of degree $1$ in $S^{\boxtimes n}$.
Let
\[\begin{aligned}
	f & =f_1\otimes \cdots\otimes f_n ,
		& \mbox{where } & f_i=\phi_{iv_1}^{\nu_1}\cdots\phi_{iv_d}^{\nu_d} \\
	g & =g_1\otimes\cdots\otimes g_n,
		& \mbox{where } & g=\gamma_{iu_1}^{\mu_1}\cdots
		\gamma_{iu_d}^{\mu_d}.
\end{aligned}\]
Note that $f$ and $g$ are both variables 
in $P_n$ with equal multi-homogeneous
degree.
Next let $\scrV=\{\nu_1,\dots,\nu_d\}$ and $\scrU=\{\mu_1,\dots,\mu_d\}$ and
\[	F_1=\prod_{\nu\in\scrV}f^{\nu} \quad\mbox{and}\quad
	F_2=\prod_{\mu\in\scrU}g^{\mu}.
\]
Here we do not take the product over all $\nu_\ell$ and $\mu_\ell$ respectively
because there may be repeated indices (i.e., $\nu_\ell=\nu_{\ell'}$) 
and we want $F'$ and $G'$ to be factors of $F$ and $G$ 
respectively. Note that
$F'$ and $G'$ are monomials in $P_n$ in 
degree $\leq d\leq d_0$. Therefore, we have binomial relations
\[
	F_1-f\cdot G_1 \quad\mbox{and}\quad F_2-g\cdot G_2
\]
in $\scrI_n$ in degrees $\leq d_0^2$ for 
appropriate monomials $G_1$ and $G_2$ in $P_n$. 
Furthermore, we have the decomposition.
\[
	F-G=H_1(F_1-f\cdot G_1)
	+(f\cdot H_1\cdot G_1-g\cdot H_2\cdot G_2) 
	+ H_2(F_2-g\cdot G_2)\in\scrI_n
\]
for appropriate monomials $H_1$ and $H_2$ in $P_n$. Therefore, 
$f\cdot H_1\cdot G_1-g\cdot H_2\cdot G_2$ is a binomial relation.
Thus the problem is reduced
to the case of binomial relations such that there is a multi-homogeneous degree $d$
variable in each term for some $d\leq d_0$.

\subsection{}\label{s n_0}
Let $\scrB$ be the subset of vectors in $\bbZ^{2r}$ which are the
linear combinations in of vectors in $\scrX^2$ such that the 
sum of the coefficients is at most $d_0$.
Using the analogous method described in \S\ref{s d_0} for $\scrB$ (analogous to
$\scrA$), we define
$n_0$ (analogous to $d_0$).

\subsection{}\label{s proof}
Fix a positive integer $n$. Let $F-G\in\scrI_n$ be a reduced binomial relation
such as that given in \eqref{relation}, but now assume that both $f=f^1$ 
and $g=g^1$ have multi-homogeneous degree $d\leq d_0$;
define matrix $M$ as in \S\ref{s reduce}.
By \S\ref{s n_0} we can partition $\{I_1,\dots,I_L\}$ of rows $[n]$ 
such that each block of rows is size $\leq n_0$ and forms a zero-sum sub-matrix
of $M$. Suppose that each block $I_\ell$ corresponds the pure tensors 
\[\begin{aligned}
	a_{I_\ell}&=a_1\otimes\cdots\otimes a_n, & 
	\mbox{where }a_i & =\begin{cases}
	1 & \mbox{if $i\notin I_\ell$} \\
	\phi_{i1} \cdots \phi_{id} & \mbox{if $i\in I_\ell$}
	\end{cases}\\
	b_{I_\ell}&=b_1\otimes \cdots\otimes b_n, &
	\mbox{where }b_i & =\begin{cases}
	1 & \mbox{if $i\notin I_\ell$} \\
	\gamma_{i1} \cdots \gamma_{id} & \mbox{if $i\in I_{\ell}$}
	\end{cases}
\end{aligned}\]
Let $m_0=a_{I_1}\cdots a_{I_L}=f$. Then for all $\ell\in[L]$ 
recursively let 
\[
	m_\ell=\frac{m_{\ell-1}}{a_{I_\ell}}b_{I_\ell}.
\]
Note that each $m_\ell$ is a variable in $P_n$ of
multi-homogeneous degree $d$ since the weight matrix
of $m_\ell$ is the weight matrix of $m_{\ell-1}$ with a zero-sum
block of rows replaced by another zero-sum block of rows, and hence
itself zero-sum.
Also notice that this sequence terminates at $m_L=b_{I_1}\cdots b_{I_L}=g$.
Since we are swapping at most $n_0 d_0$ basis elements for each $m_\ell$, 
the sub-relation that represents the swapping 
\[
	m_{\ell-1}\cdot F_\ell-m_\ell\cdot G_\ell\in\scrI_n,
\]
for appropriate monomials $F_\ell$ and $G_\ell$,
has at most degree $n_0d_0^2$.
Furthermore, we have the decomposition
\[
	F-G=H_{1}(f\cdot F_1-m_1\cdot G_1)+\cdots+H_L(m_{L-1}\cdot F_L-g\cdot G_L)
	+(g\cdot H_L\cdot F_L- G)\in\scrI_n
\]
For appropriate monomials $H_\ell$.
Since that last sub-relation can have $g$ factored out of both terms,
we have reduced the total multi-homogeneous degree of the binomial relation. Then
by induction on the total multi-homogeneous degree on relations, that last term
is generated in degrees $\leq n_0d_0^2$.
Hence we have proved that $F-G$ can be generated in degrees $\leq n_0d_0^2$. 
\end{proof}

\end{document}